\documentclass[usenames,dvipsnames, 11pt]{article}
\usepackage[utf8]{inputenc}

\usepackage[margin=2cm,a4paper]{geometry} 
\usepackage{multirow,listings,setspace,gnuplottex,latexsym,keyval,ifthen,moreverb,lscape,forest}
\usepackage{pgf,tikz}
\usetikzlibrary{arrows}

\usepackage{amsmath, amssymb, amsthm, bm}
\usepackage{color}
\usepackage{url}
\usepackage{enumerate}
\usepackage{graphicx}
\usepackage{appendix}
\usepackage{cite}
\usepackage{mathtools}
\usepackage{hyperref}

\newtheorem{thm}{Theorem}
\newtheorem{prop}[thm]{Proposition}
\newtheorem{lem}[thm]{Lemma}
\newtheorem{cor}[thm]{Corollary}

\newtheorem{rem}[thm]{Remark}

\newtheorem{ques}[thm]{Question}

\numberwithin{thm}{section}

\theoremstyle{definition}

\newcommand{\m}[1]{\mathcal{#1}}
\newcommand{\F}[2]{\mathbb{F}_{#1}^{#2}}

\DeclarePairedDelimiter{\card}{|}{|}
\DeclarePairedDelimiter{\floor}{\lfloor}{\rfloor}

\newcommand{\red}[1]{\textcolor{red}{#1}}
\newcommand{\blue}[1]{\textcolor{blue}{#1}}
\newcommand{\green}[1]{\textcolor{ForestGreen}{#1}}
\newcommand{\orange}[1]{\textcolor{orange}{#1}}

\newcommand{\cyan}[1]{\textcolor{cyan}{#1}}

\usepackage{diagbox}

\title{Subspace coverings with multiplicities}
\author{Anurag Bishnoi\thanks{Delft Institute of Applied Mathematics, Technische Universiteit Delft, 2628 CD Delft, Netherlands. E-mail: \textsf{A.Bishnoi@tudelft.nl}.} \and Simona Boyadzhiyska\thanks{Institut f\"ur Mathematik, Freie Universit\"at Berlin, 14195 Berlin, Germany.}\;\thanks{E-mail: \textsf{s.boyadzhiyska@fu-berlin.de}. Research supported by the Deutsche Forschungsgemeinschaft (DFG) Graduiertenkolleg ``Facets of Complexity'' (GRK 2434).} \and Shagnik Das\footnotemark[2]\;\thanks{E-mail: \textsf{shagnik@mi.fu-berlin.de}. Research supported by the Deutsche Forschungsgemeinschaft (DFG) project 415310276.} \and Tam\'as M\'{e}sz\'{a}ros\footnotemark[2]\;\thanks{E-mail: \textsf{tmeszaros87@gmail.com}. Research supported by  the Deutsche Forschungsgemeinschaft (DFG) under Germany’s Excellence Strategy - The Berlin Mathematics Research Center MATH+ (EXC-2046/1, project ID: 390685689).}}
\begin{document}
\maketitle

\begin{abstract}
We study the problem of determining the minimum number $f(n,k,d)$ of affine subspaces of codimension $d$ that are required to cover all points of $\mathbb{F}_2^n\setminus \{\vec{0}\}$ at least $k$ times while covering the origin at most $k - 1$ times. The case $k=1$ is a classic result of Jamison, which was independently obtained by Brouwer and Schrijver for $d = 1$. The value of $f(n,1,1)$ also follows from a well-known theorem of Alon and F\"uredi about coverings of finite grids in affine spaces over arbitrary fields. 

Here we determine the value of this function exactly in various ranges of the parameters. In particular, we prove that for $k\geq2^{n-d-1}$ we have $f(n,k,d)=2^d k-\floor*{\frac{k}{2^{n-d}}}$,
while for $n > 2^{2^d k-k-d+1}$ we have $f(n,k,d)=n + 2^d k -d-2$, and also study the transition between these two ranges. While previous work in this direction has primarily employed the polynomial method, we prove our results through more direct combinatorial and probabilistic arguments, and also exploit a connection to coding theory.
\end{abstract}


\section{Introduction} \label{sec:intro}

How many affine hyperplanes does it take to cover the vertices of the $n$-dimensional Boolean hypercube, $\{0,1\}^n$? This simple question has an equally straightforward answer --- one can cover all the vertices with a parallel pair of hyperplanes, while it is easy to see that a single plane can cover at most half the vertices, and so two planes are indeed necessary. However, the waters are quickly muddied with a minor twist to the problem.

Indeed, if one is instead asked to cover all the vertices except the origin, the parallel hyperplane construction is no longer valid. Given a moment's thought, one might come up with the much larger family of $n$ hyperplanes given by $\{ \vec{x} : x_i = 1 \}$ for $i \in [n]$. This fulfils the task and, surprisingly, turns out to be optimal, although this is far from obvious. This problem has led to rich veins of research in both finite geometry and extremal combinatorics, and in what follows we survey its history before introducing our new results.

\subsection{An origin story}

When we work over the finite field $\F{2}{}$, this problem is equivalent to the well-known blocking set problem from finite geometry, and it was in this guise that it was first studied. A blocking set in $\F{2}{n}$ is a set of points that meets every hyperplane, and the objective is to find a blocking set of minimum size. By translating, we may assume that our blocking set contains the origin $\vec{0}$, and so the problem reduces to finding a collection of points that meets all hyperplanes avoiding the origin. Applying duality, now, we return to our original problem of covering the nonzero points of $\F{2}{n}$ with affine hyperplanes.

From this perspective, there is no reason to restrict our attention to the binary field $\F{2}{}$, and we can generalise the problem to ask how many hyperplanes are needed to cover the nonzero points of $\F{q}{n}$. Going even further, one may replace the hyperplanes with affine subspaces of codimension $d$. In this generality, the problem was answered in the late 1970s by Jamison~\cite{J77}, who proved that the minimum number of affine subspaces of codimension $d$ that cover all nonzero points in $\F{q}{n}$ while avoiding the origin is $q^d - 1 + (n-d)(q-1)$. In particular, when $q = 2$ and $d=1$, this lower bound is equal to $n$, showing that the earlier construction with $n$ planes is optimal. A simpler proof of the case $d=1$ was independently provided by Brouwer and Schrijver~\cite{BS78}.

While the finite geometry motivation naturally leads one to work over finite fields, one can also study the problem over infinite fields $\F{}{}$. Of course, one would need infinitely many hyperplanes to cover all nonzero points of $\F{}{n}$, which is why we instead ask how many hyperplanes are needed to cover the nonzero points of the hypercube $\{0,1\}^n \subseteq \F{}{n}$. This problem was raised in the early 1990s by Komj\'ath~\cite{K94}, who, in order to prove some results in infinite Ramsey theory, showed that this quantity must grow with $n$. Shortly afterwards, a celebrated result of Alon and F\"uredi~\cite{AF93} established a tight bound in the more general setting of covering all but one point of a finite grid. They showed that, for any collection of finite subsets $S_1, S_2, \hdots, S_n$ of some arbitrary field $\F{}{}$, the minimum number of hyperplanes needed to cover all but one point of $S_1 \times S_2 \times \hdots \times S_n$ is $\sum_i \left( \card{S_i} - 1 \right)$. If we take $S_i = \{0,1\}$ for all $i$, this once again shows that one needs $n$ hyperplanes to cover the nonzero points of the hypercube.

\subsection{The polynomial method}

Despite these motivating applications to finite geometry and Ramsey theory, the primary reason this problem has attracted so much attention lies in the proof methods used. These hyperplane covers have driven the development of the polynomial method --- indeed, in light of his early results, this is sometimes referred to as the Jamison method in finite geometry~\cite{BF91}.

To see how polynomials come into play, suppose we have a set of hyperplanes $\{H_i : i \in [m] \}$ in $\F{}{n}$, with the plane $H_i$ defined by $H_i = \{ \vec{x} : \vec{x} \cdot \vec{a}_i = c_i \}$ for some normal vector $\vec{a}_i \in \F{}{n}$ and some constant $c_i \in \F{}{}$. We can then define the degree-$m$ polynomial $f(\vec{x}) = \prod_{i \in m} \left( \vec{x} \cdot \vec{a}_i - c_i \right)$, observing that $f(\vec{x}) = 0$ if and only if $\vec{x}$ is covered by one of the hyperplanes $H_i$. Thus, lower bounds on the degrees of polynomials that vanish except at the origin translate to lower bounds on the number of hyperplanes needed to cover all nonzero points.

This approach has proven very robust, and lends itself to a number of generalisations. For instance, K\'os, M\'esz\'aros and R\'onyai~\cite{KMR11} and Bishnoi, Clark, Potukuchi and Schmitt~\cite{BCPSch18} considered variations over rings, while Blokhuis, Brouwer and Sz\H{o}nyi~\cite{BBSz10} studied the problem for quadratic surfaces and Hermitian varieties in projective and affine spaces over $\F{q}{}$.

\subsection{Covering with multiplicity}

In this paper, we shall remain in the original setting, but instead extend the problem to higher multiplicities. That is, we shall seek the minimum number of hyperplanes needed in $\F{}{n}$ to cover the nonzero points at least $k$ times, while the origin is covered fewer times. Previous work in this direction has imposed the stricter condition of avoiding the origin altogether; Bruen~\cite{B92} considered this problem over finite fields, while Ball and Serra~\cite{BS09} and K\'os and R\'onyai~\cite{KR12} worked with finite grids over arbitrary fields, with some further generalisations recently provided by Geil and Matr\'inez-Pe\~{n}as~\cite{GM19}. In all of these papers, the polynomial method described above was strengthened to obtain lower bounds for this problem with higher multiplicities. However, these lower bounds are most often not tight; Zanella~\cite{Z02} discusses when Bruen's bound is sharp, with some improvements provided by Ball~\cite{B00}.

Significant progress in this line of research was made recently when Clifton and Huang~\cite{CH20} studied the special case of covering all nonzero points of $\{0,1\}^n \subseteq \mathbb{R}^n$ at least $k$ times, while leaving the origin uncovered. Observe that one can remove $k-1$ hyperplanes arbitrarily from such a cover, and the remainder will still cover each nonzero point at least once. Thus, by the Alon--F\"uredi Theorem, we must be left with at least $n$ planes, giving a lower bound of $n+k-1$. While it is not hard to see that this is tight for $k=2$, Clifton and Huang used Ball and Serra's Punctured Combinatorial Nullstellensatz~\cite{BS09} to improve the lower bound for larger $k$. They showed that for $k = 3$ and $n \ge 2$, the correct answer is $n + 3$, while for $k \ge 4$ and $n \ge 3$, the answer lies between $n + k + 1$ and $n + \binom{k}{2}$, conjecturing the upper bound to be correct when $n$ is large with respect to $k$. However, they showed that this was far from the case when $n$ is fixed and $k$ is large; in this range, the answer is $(c_n + o(1))k$, where $c_n$ is the $n$th term in the harmonic series.

A major breakthrough was then made by Sauermann and Wigderson~\cite{SW20}, who skipped the geometric motivation and resolved the polynomial problem directly. More precisely, they proved the following theorem.

\begin{thm} \label{thm:sauwig}
Let $k \ge 2$ and $n \ge 2k-3$, and let $P \in \mathbb{R}[x_1, \hdots, x_n]$ be a polynomial having zeroes of multiplicity at least $k$ at all points in $\{0,1\}^n \setminus \{ \vec{0} \}$, and such that $P$ does not have a zero of multiplicity at least $k-1$ at $\vec{0}$. Then $P$ must have degree at least $n + 2k - 3$. Furthermore, for every $\ell \in \{0, 1, \hdots, k-2\}$, there exists a polynomial $P$ with degree exactly $n + 2k - 3$ having zeroes of multiplicity at least $k$ at all points in $\{0,1\}^n \setminus \{ \vec{0} \}$, and such that $P$ has a zero of multiplicity exactly $\ell$ at $\vec{0}$.
\end{thm}

As an immediate corollary, this improves the lower bound in the Clifton--Huang result from $n+k+1$ to $n+2k-3$. However, Theorem~\ref{thm:sauwig} establishes that $n+2k-3$ is also an upper bound for the polynomial problem, whereas Clifton and Huang conjecture that the answer for their problem should be $n + \binom{k}{2}$. This suggests that the polynomial method alone is not sufficient to resolve the hyperplane covering problem.

Even though Theorem~\ref{thm:sauwig} is stated for polynomials defined over $\mathbb{R}$, Sauermann and Wigderson note that the proof works over any field of characteristic zero. However, the result need not hold over finite fields. In particular, they show the existence of a polynomial $P_4$ over $\F{2}{}$ of degree $n + 4$ with zeroes of multiplicity four at all nonzero points in $\F{2}{n}$ and with $P_4(\vec{0}) \neq 0$. More generally, for every $k \ge 4$, $P_k(\vec{x}) = x_1^{k-4} (x_1 - 1)^{k-4} P_4(\vec{x})$ is a binary polynomial of degree only $n + 2k - 4$ with zeroes of multiplicity $k$ at all nonzero points and of multiplicity $k-4$ at the origin. The correct behaviour of the problem over finite fields is left as an open problem. 

Note also that Theorem~\ref{thm:sauwig} allows the origin to be covered up to $k-2$ times. Sauermann and Wigderson also considered the case where the origin must be covered with multiplicity exactly $k-1$, showing that the minimum degree then increases to $n+2k-2$. In contrast to Theorem~\ref{thm:sauwig}, the proof of this result is valid over all fields.

\subsection{Our results}

In this paper, we study the problem of covering with multiplicity in $\F{2}{n}$. We are motivated not only by the body of research described above, but also by the fact, as we shall show in Proposition~\ref{prop:strictcodeequiv}, when one forbids the origin from being covered, this problem is equivalent to finding linear binary codes of large minimum distance. As this classic problem from coding theory has a long and storied history of its own, and is likely to be very difficult, we shall instead work in the setting where we require all nonzero points in $\F{2}{n}$ to be covered at least $k$ times while the origin can be covered at most $k-1$ times.

In light of the previous results, we shall abstain from employing the polynomial method, and instead attack the problem more directly with combinatorial techniques. As an added bonus, our arguments readily generalise to covering points with codimension-$d$ affine subspaces, rather than just hyperplanes, thereby extending Jamison's original results in the case $q=2$. To be able to discuss our results more concisely, we first introduce some notation that we will use throughout the paper.

Given integers $k \ge 1$ and $n \ge d \ge 1$, we say a multiset $\m H$ of $(n-d)$-dimensional affine subspaces in $\F{2}{n}$ is a \emph{$(k,d)$-cover} if every nonzero point of $\F{2}{n}$ is covered at least $k$ times, while $\vec{0}$ is covered at most $k-1$ times. We next introduce an extremal function $f(n,k,d)$, which is defined to be the minimum possible size of a $(k,d)$-cover in $\F{2}{n}$.

For instance, when we take $k = 1$, we obtain the original covering problem, and from the work of Jamison~\cite{J77} we know $f(n,1,d) = n + 2^d - d - 1$. At another extreme, if we take $d = n$, then our affine subspaces are simply individual points, each of which must be covered $k$ times, and hence $f(n,k,n) = k \left( 2^n - 1 \right)$. We study this function for intermediate values of the parameters, determining it precisely when either $k$ is large with respect to $n$ and $d$, or $n$ is large with respect to $k$ and $d$, and derive asymptotic results otherwise.


\begin{thm} \label{thm:main}
Let $k \ge 1$ and $n \ge d \ge 1$. Then:
\begin{itemize}
	\item[(a)] If $k \ge 2^{n-d-1}$, then $f(n,k,d) = 2^d k - \floor*{ \frac{k}{2^{n-d}} }$.
	\item[(b)] If $n > 2^{2^d k - d - k + 1}$, then $f(n,k,d) = n + 2^d k - d - 2$.
	\item[(c)] If $k \ge 2$ and $n \ge \floor*{\log_2 k} + d + 1$, then $n + 2^d k - d - \log_2 (2k) \le f(n,k,d) \le n + 2^d k - d - 2$.
\end{itemize}
\end{thm}

There are a few remarks worth making at this stage. First, observe that, just as in the Clifton--Huang setting, the extremal function $f(n,k,d)$ exhibits different behaviour when $n$ is fixed and $k$ is large as compared to when $k$ is fixed and $n$ is large. Second, and perhaps most significantly, Theorem~\ref{thm:main} demonstrates the gap between the hyperplane covering problem and the polynomial degree problem: our result shows that, for any $k\geq 4$ and sufficiently large $n$, we have $f(n,k,1) = n + 2k - 3$, whereas the answer to the corresponding polynomial problem is at most $n + 2k - 4$, as explained after Theorem~\ref{thm:sauwig}. Our ideas allow us to establish an even stronger separation in the case $k = 4$ --- while the polynomial $P_4$ constructed by Sauermann and Wigderson, which has zeroes of multiplicity at least four at all nonzero points of $\F{2}{n}$ while not vanishing at the origin, has degree only $n + 4$, we shall show in Corollary~\ref{cor:bound} that any hyperplane system with the corresponding covering properties must have size at least $n + \log \left( \tfrac23 n \right)$. Third, we see that in the intermediate range, when both $n$ and $k$ grow moderately, the bounds in (c) determine $f(n,k,d)$ up to an additive error of $\log_2 (2k)$, which is a lower-order term. Thus, $f(n,k,d)$ grows asymptotically like $n + 2^d k$. Last of all, if one substitutes $k = 2^{n-d-1} - 1$, the lower bound from (c) is larger than the value in (a). This shows that $k \ge 2^{n-d-1}$ is indeed the correct range for which the result in (a) is valid. In contrast, we believe the bound on $n$ in (b) is far from optimal, and discuss this in greater depth in Section~\ref{sec:mid}.

The remainder of this paper is devoted to the proof of Theorem~\ref{thm:main}, and is organised as follows. In Section~\ref{sec:largek} we prove part (a), determining the extremal function for large multiplicities. We prove part (b) in Section~\ref{sec:largen}, handling the case when the dimension of the ambient space grows quickly. A key step in the proof is showing the intuitive, yet surprisingly not immediate, fact that $f(n,k,d)$ is strictly increasing in $n$, as a result of which we shall also be able to deduce the bounds in (c). Section~\ref{sec:mid} is devoted to the study of the gradual transition between parts (a) and (b), where we exhibit some constructions that show $f(n,k,d)$ takes values strictly between those of parts (a) and (b). Finally, we end by presenting some concluding remarks and open problems in Section~\ref{sec:conc}.


\section{Covering with large multiplicity}
\label{sec:largek}

In this section we prove Theorem~\ref{thm:main}(a), handling the case of large multiplicities. We start by introducing some definitions and notation that we will use in the proof. To start with, it will be convenient to have some notation for affine hyperplanes. Given a nonzero vector $\vec{u}\in \F{2}{n}$, let $H_{\vec{u}}$ denote the hyperplane $\{ \vec{x} : \vec{x} \cdot \vec{u} = 1 \}$. 

Next, it will sometimes be helpful to specify how many times the origin is covered. Hence, given integers $n \ge d \ge 1$ and $k > s \ge 0$, we call a $(k,d)$-cover in $\F{2}{n}$ a \emph{$(k,d;s)$-cover} if it covers the origin exactly $s$ times. Let us write $g(n,k,d;s)$ for 
the minimum possible size of a $(k,d;s)$-cover and call a cover \emph{optimal} if it has this minimum size. Clearly, we have $f(n,k,d)=\min_{0\leq s<k} g(n,k,d;s)$, so any knowledge about this more refined function directly translates to our main focus of interest.

\subsection{The lower bound}
\label{sec:counting_lower}
To start with, we prove a general lower bound, valid for all choices of parameters, that follows from a simple double-counting argument. This establishes the lower bound of Theorem~\ref{thm:main}(a).

\pagebreak

\begin{lem}\label{lem:double-count}
Let $n,k,d,s$ be integers such that $n \ge d \ge 1$ and $k > s \ge 0$. Then
\begin{equation*}
g(n,k,d;s) \geq 2^d k - \left\lfloor\frac{k-s}{2^{n-d}}\right\rfloor.
\end{equation*}
In particular, $f(n,k,d) \ge 2^d k - \floor*{\frac{k}{2^{n-d}}}$.
\end{lem}

\begin{proof}
Let $\m H$ be an optimal $(k,d;s)$-cover of $\F{2}{n}$, so that we have $g(n,k,d;s) = \card{\m H}$. We double-count the 
pairs $(\vec{x}, S)$ with $\vec{x} \in \F{2}{n} $, $S \in \m H$, and $\vec{x} \in S$. On the one hand, every affine subspace $S \in \m H$ contains $2^{n-d}$ points, and so there are $2^{n-d} \card {\m H}$ such pairs. On the other hand, since every nonzero point is covered at least $k$ times and the origin is covered $s$ times, there are at least $(2^n - 1)k+s$ such pairs. Thus $(2^n - 1 )k + s \le 2^{n-d} \card{\m H}$, and the claimed lower bound follows from solving for $\card{\m H}$ and observing that $g(n,k,d;s)$ is an integer. The bound on $f(n,k,d)$ is obtained by noticing that our lower bound on $g(n,k,d;s)$ is increasing in $s$, and is therefore minimised when $s = 0$.
\end{proof}

\subsection{The upper bound construction}
\label{sec:upper}

To prove the upper bound of Theorem~\ref{thm:main}(a), we must construct small $(k,d)$-covers. As a first step, we introduce a recursive method
for $(k,d;s)$-covers that allows us to reduce to the $d = 1$ case. 

\begin{lem}\label{lem:hyperplanerecursion}
For integers $n \ge d \ge 2$ and $k > s \ge 0$ we have
\begin{equation*}
g(n,k,d;s)\leq g(n-d+1,k,1;s)+2k(2^{d-1}-1),
\end{equation*}
and, therefore,
\begin{equation*}
    f(n,k,d) \le f(n-d+1,k,1) + 2k(2^{d-1}-1).
\end{equation*}
\end{lem}

\begin{proof}
We first deduce the recursive bound on $g(n,k,d;s)$. Let $S_0 \subset \F{2}{n}$ be an arbitrary $(n - d + 1)$-dimensional (vector) subspace, and let $S_1, \hdots, S_{2^{d-1} - 1}$ be its affine translates, that, together with $S_0$, partition $\F{2}{n}$. For every $1\leq i\leq 2^{d-1}-1$, partition $S_i \cong \mathbb{F}_2^{n - d+ 1}$ further into two subspaces, thereby obtaining a total of $2(2^{d - 1} - 1)$ affine subspaces of dimension $n - d$. We start by taking $k$ copies of each of these affine subspaces. This gives us a multiset of $2k(2^{d - 1} - 1)$ subspaces, which cover every point outside $S_0$ exactly $k$ times and leave the points in $S_0$ completely uncovered.

It thus remains to cover the points within $S_0$ appropriately. Since $(n-d)$-dimensional subspaces have relative codimension $1$ in $S_0$, this reduces to finding a $(k,1;s)$-cover within $S_0 \cong \F{2}{n-d+1}$. By definition, we can find such a cover consisting of ${g(n-d+1,k,1;s)}$ subspaces. Adding these to our previous multiset gives a $(k,d;s)$-cover of $\F{2}{n}$ of size $g(n-d+1,k,1;s) + 2k(2^{d-1} - 1)$, as required.

To finish, since $f(n,k,d)=\min_s g(n,k,d;s)$, and the recursive bound holds for each $s$, it naturally carries over to the function $f(n,k,d)$, giving $f(n,k,d)\leq f(n-d+1,k,1)+2k(2^{d-1}-1)$.
\end{proof}

Armed with this preparation, we can now resolve the problem for large multiplicities.

\begin{proof}[Proof of Theorem~\ref{thm:main}(a)]
The requisite lower bound, of course, is given by Lemma~\ref{lem:double-count}. 

For the upper bound, we start by reducing to the case $d = 1$. Indeed, suppose we already know the bound for $d=1$; that is, $f(n,k,1) \le 2k - \floor*{ \frac{k}{2^{n-1}}}$ for all $k \ge 2^{n-2}$. Now, given some $n \ge d \ge 2$ and $k \ge 2^{n-d-1}$, by Lemma~\ref{lem:hyperplanerecursion} we have
\[ f(n,k,d) \le f(n-d+1,k,1) + 2k(2^{d-1} - 1) \le 2k - \floor*{\frac{k}{2^{n-d+1-1}}} + 2k(2^{d-1} - 1) = 2^d k - \floor*{ \frac{k}{2^{n-d}}}, \]
as required.

Hence, it suffices to prove the bound in the hyperplane case. We begin with the lowest multiplicity covered by part (a), namely $k = 2^{n-2}$. Consider the family $\m H_0 = \{ H_{\vec{u}} : \vec{u} \in \F{2}{n}, u_n = 1 \}$, where we recall that $H_{\vec{u}} = \{ \vec{x} : \vec{x} \cdot \vec{u} = 1 \}$. Note that we then have $\card{\m H_0} = 2^{n-1} = 2k = 2k - \floor*{\frac{k}{2^{n-1}}}$, and none of these hyperplanes covers the origin. Given nonzero vectors $\vec{x} = (\vec{x}', x)$ and $\vec{u} = (\vec{u}',1)$ with $\vec{x}', \vec{u}' \in \F{2}{n-1}$ and $x \in \F{2}{}$, we have $\vec{x} \cdot \vec{u} = 1$ if and only if $\vec{x}' \cdot \vec{u}' = 1-x$. 
If $\vec{x}' \neq \vec{0}$, precisely half of the choices for $\vec{u}'$ satisfy this equation; if 
$\vec{x}' = \vec{0}$ (and thus necessarily $x = 1$), the equation is satisfied by all choices of $\vec{u}'$. 
Thus each nonzero point is covered at least $2^{n-2}$ times, 
and hence $\m H_0$ is a $(2^{n-2},1)$-cover of the desired size.

To extend the above construction to the range $2^{n-2}\leq k<2^{n-1}$, one can simply add an arbitrary choice of $k-2^{n-2}$ pairs of parallel hyperplanes. The resulting family will have $2^{n-1}+2\left(k-2^{n-2}\right)=2k=2k - \floor*{\frac{k}{2^{n-1}}}$
elements, every nonzero point is covered at least $k$ times, and the origin is covered $k-2^{n-2}<k$ times.

Finally, suppose $k\geq 2^{n-1}$. Then we can write $k = a2^{n-1}+b$ for some $a\geq 1$ and $0\leq b < 2^{n-1}$. We take $\m H_1 = \{ H_{\vec{u}} : \vec{u} \in \F{2}{n} \setminus \{\vec{0}\} \}$ to be the set of all affine hyperplanes avoiding the origin, of which there are $2^n-1$. Moreover, for each nonzero $\vec{x}$, there are exactly $2^{n-1}$ vectors $\vec{u}$ with $\vec{x} \cdot \vec{u} = 1$, and so each such point is covered $2^{n-1}$ times by the hyperplanes in $\m H_1$.

Now let $\m H$ be the multiset of hyperplanes obtained by taking $a$ copies of $\m H_1$ and appending an arbitrary choice of $b$ pairs of parallel planes. Each nonzero point is then covered $a2^{n-1} + b = k$ times, while the origin is only covered $b < 2^{n-1} \le k$ times, and so $\m H$ is a $(k,1)$-cover. Thus, 
\[ f(n,k,1) \le \card{\m H} = a(2^n - 1) + 2b = 2(a 2^{n-1} + b) - a = 2k - \floor*{ \frac{k}{2^{n-1}}}, \]
proving the upper bound.
\end{proof}


\section{Covering high-dimensional spaces}
\label{sec:largen}

In this section we turn our attention to the case when $n$ is large with respect to $k$, with the aim of proving part (b) of Theorem~\ref{thm:main}. Furthermore, the results we prove along the way will allow us to establish the bounds in part (c) as well.

\subsection{The upper bound construction}

In this range, in contrast to the large multiplicity setting, it is the upper bound that is straightforward. This bound follows from the following construction, which is valid for the full range of parameters.

\begin{lem}\label{lem:genupperbound}
Let $n,k,d$ be positive integers such that $n \ge d \ge 1$ and $k \ge 2$. Then
\begin{equation*}
f(n,k,d)\leq n + 2^d k - d - 2.
\end{equation*}
\end{lem}
\begin{proof}
We start by resolving the case $d=1$ and $k=2$, for which we consider the family of hyperplanes $\m H = \{ H_{\vec{e}_i} : i\in[n] \}\cup \{H_{\vec{1}}\}$, where $\vec{e}_i$ is the $i$th standard basis vector and $\vec{1}$ is the all-one vector. To see that this is a $(2,1$)-cover of $\F{2}{n}$, note first that the planes all avoid the origin. Next, if we have a nonzero vector $\vec{x}$, it is covered by the hyperplanes $\{ H_{\vec{e}_i} : i\in[n] \}$ as many times as it has nonzero entries. Thus, all vectors of Hamming weight at least two are covered twice or more. 
The only remaining vectors are those of weight one, which are covered once by $\{ H_{\vec{e}_i} : i\in[n] \}$, but these are all covered for the second time by $H_{\vec{1}}$. Hence $\m H$ is indeed a $(2,1)$-cover, and is of the required size, namely $n+1$.

Now we can extend this construction to the case $d=1$ and $k\geq 3$ by simply adding $k-2$ arbitrary pairs of parallel hyperplanes. The resulting family will be a $(k,1;k-2)$-cover (and hence, in particular, a $(k,1)$-cover) of size $n+2k-3$, matching the claimed upper bound.

That leaves us with the case $d \ge 2$, which we can once again handle by appealing to Lemma~\ref{lem:hyperplanerecursion}. In conjunction with the above construction, we have
\begin{equation*}
f(n,k,d)\leq f(n-d+1,k,1)+2k(2^{d - 1} - 1)\leq  n-d+1 +2k-3 + 2k(2^{d - 1} - 1),
\end{equation*}
which simplifies to the required $n + 2^d k - d - 2$.
\end{proof}

\subsection{Recursion, again}

The upper bound in Lemma~\ref{lem:genupperbound} is strictly increasing in $n$. Our next step is to show that this behaviour is necessary --- that is, the higher the dimension, the harder the space is to cover. Although intuitive, this fact turned out to be less elementary than expected, and our proof makes use of the probabilistic method.

\begin{lem} \label{lem:recursive}
Let $n,k,d,s$ be integers such that $n\ge 2$, $n \ge d \ge 1$, and $k > s \ge 0$. Then
\[ g(n,k,d;s) \ge g(n-1,k,d;s) + 1. \]
\end{lem}

\begin{proof}
Let $\m H$ be an optimal $(k,d;s)$-cover of $\F{2}{n}$. To prove the lower bound on its size, we shall construct from it a $(k,d;s)$-cover $\m H'$ of $\F{2}{n-1}$, which must comprise of at least $g(n-1,k,d;s)$ subspaces. To obtain this cover of a lower-dimensional space, we restrict $\m H$ to a random hyperplane $H \subset \F{2}{n}$ that passes through the origin. Since $\m H$ is a $(k,d;s)$-cover of all of $\F{2}{n}$, it certainly covers $H \cong \F{2}{n-1}$ as well.

However, we require $\m H'$ to be a $(k,d;s)$-cover of $H$, which must be built of affine subspaces of codimension $d$ relative to $H$ --- that is, subspaces of dimension one less than those in $\m H$. Fortunately, when intersecting the subspaces $S \in \m H$ with a hyperplane, we can expect their dimension to decrease by one. The exceptional cases are when $S$ is disjoint from $H$, or when $S$ is contained in $H$. In the former case, $S$ does not cover any points of $H$, and can therefore be discarded from $\m H'$. In the latter case, we can partition $S$ into two subspaces $S = S_1 \cup S_2$, where each $S_i$ is of codimension $d$ relative to $H$, and replace $S$ with $S_1$ and $S_2$ in $\m H'$. By making these changes, we obtain a family $\m H'$ of codimension-$d$ subspaces of $H$. Moreover, these subspaces cover the points of $H$ exactly as often as those of $\m H$ do, and thus $\m H'$ is a $(k,d;s)$-cover of $H$.

When building this cover, though, we need to control its size. Let $X$ denote the set of subspaces $S \in \m H$ that are disjoint from $H$, and let $Y$ denote the set of subspaces $S \in \m H$ that are contained in $H$. We then have $\card{ \m H'} = \card{\m H} - \card{X} + \card{Y}$. The objective, then, is to show that there is a choice of hyperplane $H$ for which $\card{X} > \card{Y}$, in which case the cover $\m H'$ we build is relatively small.

Recall that $H$ was a random hyperplane in $\F{2}{n}$ passing through the origin, which is to say it has a normal vector $\vec{u}$ chosen uniformly at random from $\F{2}{n}\setminus \{\vec{0}\}$. To compute the expected sizes of $X$ and $Y$, we consider the probability that a subspace $S \in \m H$ is either disjoint from or contained in $H$.

Let $S\in \m H$ be arbitrary and suppose first that $\vec{0}\in S$. We immediately have $\mathbb{P}(S\in X)=0$, as in this case $\vec{0}\in S\cap H$, so $S$ and $H$ cannot be disjoint. On the other hand, $\mathbb{P}(S\in Y)=\frac{2^d-1}{2^{n}-1}$, as we have
$S\subseteq H$ exactly when the normal vector $\vec{u}$ is a nonzero element of the $d$-dimensional orthogonal
complement, $S^{\perp}$, of $S$ in $\F{2}{n}$.

In the other case, when $\vec{0}\notin S$, we can write $S$ in the form $T+\vec{v}$, where $\vec{0} \in T \subset \F{2}{n}$ is an $(n-d)$-dimensional vector subspace and $\vec{v} \in \F{2}{n}\setminus T$. Then $S$ is disjoint from $H$ if and only if $\vec{u} \in S^{\perp}$ and $\vec{u} \cdot \vec{v} = 1$. Since $\vec{v} \notin T$, these are independent conditions, and so we have $\mathbb{P}(S \in X)=\frac{2^{d-1}}{2^{n}-1}$. Similarly, in order to have $S \subseteq H$, $\vec{u}$ must be a nonzero vector satisfying $\vec{u} \in S^{\perp}$ and $\vec{u} \cdot \vec{v} = 0$, and so $\mathbb{P}(S\in Y)=\frac{2^{d-1}-1}{2^{n}-1}$.

Now, using linearity of expectation, we have
\begin{align*}
\mathbb{E} \left[ \card{X} - \card{Y} \right]
&= \sum_{S\in \m H} \left(\mathbb{P}(S\in X)-\mathbb{P}(S\in Y)\right)\\
&= \sum_{S\in \m H: \vec{0}\notin S} \left(\frac{2^{d-1}}{2^{n}-1}-\frac{2^{d-1}-1}{2^{n}-1}\right) + \sum_{S\in \m H: \vec{0}\in S} \left(0-\frac{2^d-1}{2^{n}-1}\right) \\
&=\frac{\card{\{S\in \m H: \vec{0}\notin S\}}-\left(2^d-1\right)\card{\{S\in \m H: \vec{0}\in S\}}}{2^n-1}
=\frac{\card{\m H} - 2^d s}{2^n-1},
\end{align*}
where we used the fact that $\m H$ is a $(k,d;s)$-cover, and thus $\card{\{S \in \m H: \vec{0} \in S \}} = s$. We now apply the lower bound on $\card{\m H}$ given by Lemma~\ref{lem:double-count} to obtain
\[ \mathbb{E} \left[ \card{X} - \card{Y} \right] \geq \frac{2^d k - \floor*{\frac{k-s}{2^{n-d}}} -2^d s}{2^n-1} = \frac{2^d (k-s) - \floor*{\frac{k-s}{2^{n-d}}}}{2^n-1}>0. \]
Therefore, there must be a hyperplane $H$ for which $\card{X} - \card{Y} \ge 1$. The corresponding cover of $H$ thus has size at most $\card{\m H} - 1$ but, as a $(k,d;s)$-cover of an $(n-1)$-dimensional space, has size at least $g(n-1,k,d;s)$. This gives $\card{\m H} - 1 \ge \card{\m H'} \ge g(n-1,k,d;s)$, whence the required bound, $g(n,k,d;s) = \card{\m H} \ge g(n-1,k,d;s) + 1$.
\end{proof}

While this inequality will be used in our proof of part (b) of Theorem~\ref{thm:main}, it also gives us what we need to prove the bounds in part (c).

\begin{proof}[Proof of Theorem~\ref{thm:main}(c)]
Lemma~\ref{lem:genupperbound} gives us the upper bound, $f(n,k,d) \le n + 2^d k - d -2$, which is in fact valid for all $k \ge 2$ and $n \ge d \ge 1$.

When $n \ge \floor*{\log_2 k} + d + 1$, we can prove the lower bound, $f(n,k,d) \ge n + 2^d k - d - \log_2 (2k)$, by induction on $n$. For the base case, when $n = \floor*{\log_2 k} + d + 1$, we appeal to Lemma~\ref{lem:double-count}, which gives
\[ f(n,k,d) \ge 2^d k - \floor*{\frac{k}{2^{n-d}}} = 2^d k = n + 2^d k - d - \floor*{\log_2 k} - 1 \ge n + 2^d k - d - \log_2 (2k). \]

For the induction step we appeal to Lemma~\ref{lem:recursive}. First note that the lemma gives $f(n,k,d) = \min_s g(n,k,d;s) \ge \min_s \left( g(n-1,k,d;s) + 1 \right) = f(n-1,k,d) + 1$. Thus, using the induction hypothesis, for all $n > \floor*{\log_2 k} + d + 1$ we have
\[ f(n,k,d) \ge f(n-1,k,d) + 1 \ge n-1 + 2^d k - d - \log_2(2k) + 1 = n + 2^d k - d - \log_2 (2k), \]
completing the proof.
\end{proof}

\subsection{A coding theory connection}\label{sec:codingconnection}

In Lemma~\ref{lem:recursive}, we proved a recursive bound on $g(n,k,d;s)$ that is valid for all values of $s$, the number of times the origin is covered. In this subsection, we establish the promised connection to coding theory, which is the key to our proof. Indeed, as observed in Corollary~\ref{cor:highcovering} below, it allows us to restrict our attention to only two feasible values of $s$.

We begin with $(k,1;0)$-covers of $\F{2}{n}$, showing that, in this binary setting, hyperplane covers that avoid the origin are in direct correspondence with linear codes of large minimum distance.

\begin{prop} \label{prop:strictcodeequiv}
A $(k,1;0)$-cover of $\F{2}{n}$ of cardinality $m$ is equivalent to an $n$-dimensional linear binary code of 
length $m$ and minimum distance at least $k$.
\end{prop}

\begin{proof}
Let $\m H = \{H_1, H_2, \hdots, H_m\}$ be a $(k,1;0)$-cover of $\F{2}{n}$. Since none of the hyperplanes cover 
the origin, for each $i \in [m]$, $H_i$ has to be described by the equation  $\vec{u}_i\cdot \vec{x} = 1$ for some 
$\vec{u}_i \in \F{2}{n} \setminus \{ \vec{0} \}$. Let $A$ be the $m \times n$ matrix whose rows are 
$\vec{u}_1, \vec{u}_2, \hdots, \vec{u}_m$. We claim that $A$ is the generator matrix of a linear binary code of dimension $n$, length $m$ and minimum distance at least $k$. Since each $\vec{x} \in \F{2}{n} \setminus \{ \vec{0} \}$ is covered by at least $k$ of the planes, it follows that the vector $A \vec{x}$ has weight at least $k$, which in turn is equivalent to the vectors in the column space of $A$ having minimum distance at least $k$. Indeed, any vector $\vec{y}$ in the column space can be expressed in the form $A \vec{w}$ for some $\vec{w} \in \F{2}{n}$. Thus, given two vectors $\vec{y}_1, \vec{y}_2$ in the column space, their difference is of the form $A(\vec{w}_1 - \vec{w}_2)$, where $\vec{x} = \vec{w}_1 - \vec{w}_2$ is nonzero. Hence this difference has weight at least $k$; i.e., the two vectors $\vec{y}_1$ and $\vec{y}_2$ have distance at least $k$.

Conversely, given a linear binary code of dimension $n$, length $m$ and minimum distance at least $k$, let $\vec{u}_1, \vec{u}_2, \hdots, \vec{u}_m$ be the rows of the generator matrix. By the same reasoning as above, the hyperplanes $H_i$, $i\in [m]$, defined by the equation $\vec{u}_i\cdot \vec{x} = 1$, form a $(k,1;0)$-cover of $\F{2}{n}$.
\end{proof}

Thus, the problem of finding a small $(k,1;0)$-cover of $\F{2}{n}$ corresponds to finding an $n$-dimensional linear code of minimum distance at least $k$ and small length. This is a central problem in coding theory and, as such, has been extensively studied. We can therefore leverage known bounds to bound the function $g(n,k,1;0)$. 

\begin{cor} \label{cor:bound}
For all $k \geq 2$ and $n \geq 1$,
\[ g(n, k,1;0)\geq n + \left\lfloor \frac{k - 1}{2} \right\rfloor \log \left( \frac{2n}{k-1} \right). \]
\end{cor}

\begin{proof}
Let $\m H$ be an optimal $(k,1,0)$-cover and let $\m C \subseteq \F{2}{n}$ be the equivalent $n$-dimensional linear binary code of length $m=\card{\m H}$ and minimum distance at least $k$, as described in Proposition~\ref{prop:strictcodeequiv}. We can now appeal to the Hamming bound: since the code has minimum distance $k$, the balls of radius $t = \floor*{\frac{k-1}{2}}$ around the $2^n$ points of $\m C$ must be pairwise disjoint. As each ball has size $\sum_{i=0}^t \binom{m}{i}$, and the ambient space has size $2^m$, we get
\begin{equation*}
2^n\leq \frac{2^m}{\sum_{i=0}^t \binom{m}{i}}.
\end{equation*}
We bound the denominator from below by
\begin{equation*}
\sum_{i=0}^t \binom{m}{i}\geq \binom{m}{t} \ge \left( \frac{m}{t} \right)^t \ge \left( \frac{n}{t} \right)^t = 2^{t \log \tfrac{n}{t}},
\end{equation*}
where the last inequality is valid provided $m\geq n$, as it must be. Thus we conclude
\begin{equation*}
g(n,k,1;0)=|\m H|=m \ge n + t \log \tfrac{n}{t}\geq n + \left\lfloor \frac{k - 1}{2} \right\rfloor \log \left( \frac{2n}{k-1} \right). \qedhere
\end{equation*}
\end{proof}

\begin{rem} \label{rem:gvbound}
Although it may seem that some of our bounds might be wasteful, one can deduce upper bounds from the Gilbert-Varshamov bound, which is obtained by considering a random linear code. In particular, if $n$ is large with respect to $k$, one finds that $g(n, k,1;0) \le n + (k - 1) \log(2n)$. 
Narrowing the gap between these upper and lower bounds remains an active area of research in coding theory. 
\end{rem}

The above lower bound can be used to show that if $n$ is large with respect to $k$ and $d$ then every optimal $(k,d)$-cover has to cover the origin many times. This corollary is critical to our proof of the upper bound. 

\begin{cor}\label{cor:highcovering}
If $n >2^{2^d k -k-d+ 1}$ then any optimal $(k,d)$-cover of $\mathbb{F}_2^n$ covers the origin at least $k-2$ times.
\end{cor}
\begin{proof}
Let $S_1, \dots, S_m$ be an optimal $(k,d)$-cover, and, if necessary, relabel the subspaces so that $S_1, \dots, S_s$ are the affine subspaces covering the origin. Suppose for a contradiction that $s \leq k - 3$, and observe that if we delete the first $k-3$ subspaces, each nonzero point must still be covered at least thrice, while the origin is left uncovered. That is, $S_{k-2}, S_{k-1}, \hdots, S_m$ forms a $(3,d;0)$-cover of $\F{2}{n}$. 

For each $k-2 \le j \le m$, we can then extend $S_j$ to an arbitrary hyperplane $H_j$ that contains $S_j$ and avoids the origin. Then $\{H_{k-2}, H_{k-1}, \hdots, H_m\}$ is a $(3,1;0)$-cover, and hence $m-k+3 \ge g(n,3,1;0)$. 

By Corollary~\ref{cor:bound}, this, together with the assumption $n >2^{2^d k -k-d+ 1}$, implies
\begin{equation*}
f(n,k,d)=m\geq g(n,3,1;0)+k-3\geq n +\log n + k-3>n + 2^d k -k-d+ 1+k-3= n + 2^d k - d - 2,
\end{equation*}
which contradicts the upper bound from Lemma~\ref{lem:genupperbound}.
\end{proof}

\begin{rem}
Observe that Corollary~\ref{cor:highcovering} in fact gives us some stability for large dimensions. If $n = 2^{2^d k - k - d + \omega(1)}$, then the above calculation shows that any $(k,d)$-cover that covers the origin at most $k-3$ times has size at least $n + 2^d k + \omega(1)$. Thus, when $n = 2^{2^d k - k - d + \omega(1)}$, any $(k,d)$-cover that is even close to optimal must cover the origin at least $k-2$ times.
\end{rem}

\subsection{The lower bound}

By Corollary~\ref{cor:highcovering}, when trying to bound $f(n,k,d) = \min_s g(n,k,d;s)$ for large $n$, we can restrict our attention to $s \in \{k-2,k-1\}$. First we deal with the latter case.

\begin{lem}\label{lem:k-1}
Let $n,k,d$ be positive integers such that $n \ge d \ge 1$. Then
\begin{equation*}
g(n,k,d;k-1)= n + 2^d k - d - 1.
\end{equation*}
\end{lem}
\begin{proof}
To prove the statement, we will show that, for all positive integers $n,k,d$ with $n \ge d \ge 1$, we have $g(n+1,k,d;k-1) = g(n,k,d;k-1)+1$. Combined with the simple observation that $g(d,k,d;k-1) =2^d k - 1$ for all $k \ge 1$, since when $d = n$ we are covering with individual points, this fact will indeed imply the desired result. 

By Lemma~\ref{lem:recursive} we know that $g(n+1,k,d;k-1) \geq g(n,k,d;k-1)+1$. For the other inequality, consider an optimal $(k,d;k-1)$-cover $\mathcal{H}$ of $\mathbb{F}_2^n$. For every $S\in \m H$, let $S'=S\times \{0,1\}$, which is a codimension-$d$ affine subspace of $\F{2}{n+1}$, and let $S_0$ be any $(n + 1 - d)$-dimensional affine subspace of $\F{2}{n+1}$ that contains the vector $(0,\hdots,0,1)$ but avoids the
origin. We claim that $\m H'=\{ S' : S\in \m H\}\cup\{S_0\}$ is a $(k,d;k-1)$-cover of $\F{2}{n+1}$. 
Indeed, for all $S\in \m H$, a point of the form $(\vec{x},t)$ is covered by $S'$ if and only if $\vec{x}$ 
is covered by $S$. Hence, the collection $\{S' : S\in \m H\}$ covers $\vec{0}$ exactly $k-1$ 
times and each point of the form $(\vec{x},t)$ with $\vec{x} \neq \vec{0}$ at least $k$ times. Finally, 
the point $(\vec{0},1)$ is covered $k-1$ times by the $\{S' : S\in \m H\}$ and once by the 
subspace $S_0$, so it is also covered the correct number of times. Hence $\m H'$ is indeed a
$(k,d;k-1)$-cover of of size $|\m H|+1$, and so the second inequality follows.
\end{proof}

\begin{rem}
Recall that the special case of $d = 1$, $g(n, k, 1; k - 1) = n + 2k - 2$, also follows from~\cite[Theorem 1.5]{SW20}. 
\end{rem}

The proof of Theorem~\ref{thm:main}(b) is now straightforward.

\begin{proof}[Proof of Theorem~\ref{thm:main}(b)]
The upper bound is given by Lemma~\ref{lem:genupperbound}. For the lower bound, first observe that for any valid choice of the parameters, we have $g(n,k,d;s+1)\leq g(n,k,d;s)+1$, as adding any subspace
containing the origin to a $(k,d;s)$-cover yields a $(k,d;s+1)$-cover. Then, by Corollary~\ref{cor:highcovering} and Lemma~\ref{lem:k-1}, we obtain
\begin{equation*}
f(n,k,d)=\min \{g(n,k,d;k-2),g(n,k,d;k-1)\}\geq g(n,k,d;k-1)-1=n + 2^d k - d - 2,
\end{equation*}
as desired.
\end{proof}


\section{The transition}
\label{sec:mid}

Parts (a) and (b) of Theorem~\ref{thm:main} determine the function $f(n,k,d)$ exactly in the two extreme ranges of the parameters --- when $k$ is exponentially large with respect to $n$, and when $n$ is exponentially large with respect to $k$. As remarked upon in the introduction, we know that in the former case, the bound on $k$ is best possible. However, that is not true for part (b), and we believe the upper bound of Lemma~\ref{lem:genupperbound} should be tight for much smaller values of $n$ as well.

In this section we explore the transition between these two ranges, with an eye towards better understanding when this upper bound becomes tight. As we saw in Lemma~\ref{lem:hyperplanerecursion}, for our upper bounds we can generally reduce to the hyperplane setting, and so we shall focus on the $d=1$ case in this section. To simplify notation, we will refer to a $(k,1)$-cover as a $k$-cover and write $f(n,k)$ instead of $f(n,k,1)$.

In this hyperplane setting, the upper bound of Lemma~\ref{lem:genupperbound}, valid for all $n \ge 1$ and $k \ge 2$, has the simple form $n + 2k - 3$. Given some fixed $k$, suppose the bound is tight for some $n_0$; that is, $f(n_0,k) = n_0 + 2k - 3$. The recursion of Lemma~\ref{lem:recursive} implies $f(n,k) \ge f(n-1,k) + 1$ for all $n \ge 2$, and so these two bounds together imply $f(n,k) = n + 2k - 3$ for all $n \ge n_0$. Hence, for every $k$, there is a well-defined threshold $n_0(k)$ such that $f(n,k) = n + 2k - 3$ if and only if $n \ge n_0(k)$. Theorem~\ref{thm:main}(b) shows $n_0(k) \le 2^k + 1$, and our goal now is to explore the true behaviour of this threshold.

\subsection{The diagonal case}

As a natural starting point, one might ask what lower bound we can provide for $n_0(k)$. From our previous results, in particular Theorem~\ref{thm:main}(a), we have seen that $f(n,k)$ behaves differently when $k$ is large compared to $n$. We therefore know the upper bound of Lemma~\ref{lem:genupperbound} is not tight when $k \ge 2^{n-2}$ or, equivalently, we know $n_0(k) > \log_2 k + 2$. However, the following construction, valid when $k \ge 4$, shows that we can improve upon Lemma~\ref{lem:genupperbound} for considerably larger values of $n$ as well.

\begin{prop} \label{prop:g2diagonal}
For all $k \ge 4$, we have $f(k,k) \le 3k - 4$. As a consequence, $n_0(k) \ge k+1$.
\end{prop}

\begin{proof}
To prove the upper bound, we must construct a $k$-cover $\m H$ of $\F{2}{k}$ of size $3k-4$. Letting $\vec{e}_i$ denote the $i$th standard basis vector and $\vec{1}$ the all-one vector, we take $\m H = \m H_1 \cup \m H_2 \cup \m H_3$, where $\m H_1 = \big\{ H_{\vec{e}_i} : i \in [k] \big\}$, $\m H_2 = \big\{ H_{\vec{1}-\vec{e}_i} : i \in [k] \big\}$, and $\m H_3$ consists of $k-4$ copies of the hyperplane with equation $\vec{x}\cdot \vec{1}=0$. Then $\m H$ has size $3k-4$, while the only planes containing the origin are those in $\m H_3$. Thus it only remains to verify that each nonzero point is covered at least $k$ times.

Given a nonzero point $\vec{x}$, let its weight be $w$. We then see that $\vec{x}$ is covered $w$ times by the planes in $\m H_1$. Next, observe that $\vec{x} \cdot \left( \vec{1} - \vec{e}_i \right)$ is equal to $w$ if $x_i = 0$, and is equal to $w - 1$ otherwise. Hence, if $w$ is odd, then $\vec{x}$ is covered by $k - w$ planes in $\m H_2$, and is thus covered at least $k$ times by $\m H$.

On the other hand, if $w$ is even, then $\vec{x}$ is covered $w$ times by the planes in $\m H_2$. However, in this case $\vec{x} \cdot \vec{1} = 0$, and so $\vec{x}$ is covered $k-4$ times by $\m H_3$ as well. In total, then, $\vec{x}$ is covered $2w + k - 4$ times. As $\vec{x}$ is a nonzero vector of even weight, we must have $w \ge 2$, and hence $\vec{x}$ is covered at least $k$ times in this case as well.

In conclusion, we see that $\m H$ forms a $k$-cover of $\F{2}{k}$, and thus $f(k,k) \le \card{\m H} = 3k - 4$. As this is smaller than the upper bound of Lemma~\ref{lem:genupperbound}, it follows that $n_0(k) \ge k+1$.
\end{proof}

\subsection{Initial values}

This still leaves us with a large range of possible values for $n_0(k)$: our lower bound is linear, while our upper bound is exponential. To get a better feel for which bound might be nearer to the truth, we next decided to take a closer look at $f(n,k)$ for small values of the parameters.

To be able to compute a number of these values efficiently, it helped to appeal to our recursive bounds. Lemma~\ref{lem:recursive} already restricts the behaviour of $f(n,k)$ as $n$ changes, showing that the function must be strictly increasing in $n$. It is also very helpful to understand how $f(n,k)$ responds to changes in $k$: as the following lemma shows, there is even less flexibility here.

\begin{lem} \label{lem:krecursion}
For all $n \ge 1$ and $k \ge 2$ we have $f(n,k-1) + 1 \le f(n,k) \le f(n,k-1) + 2$.
\end{lem}

\begin{proof}
For the lower bound, observe that, given a $k$-cover of size $f(n,k)$, removing a hyperplane covering the origin (or, if no such plane exists, an arbitrary plane) leaves us with a $(k-1)$-cover, and thus $f(n,k-1) \le f(n,k) -1$.

For the upper bound, given a $(k-1)$-cover of size $f(n,k-1)$, we can add an arbitrary pair of parallel hyperplanes to obtain a $k$-cover. Thus $f(n,k) \le f(n,k-1) + 2$.
\end{proof}

Thus, if we know the value of $f(n,k-1)$, there are only two possible values for $f(n,k)$. This becomes even more powerful when used in combination with Lemma~\ref{lem:recursive}, which guarantees $f(n,k) \ge f(n-1,k) + 1$. Hence, in case we have $f(n-1,k) = f(n,k-1) + 1$, the only possible value for $f(n,k)$ is $f(n,k-1) + 2$.

Although this may seem a very conditional statement, this configuration occurs quite frequently, as one can see in Table~\ref{tab:uptosix} below, and allows us to deduce several values of $f(n,k)$ for free. This observation, together with our previous bounds (and noting that $f(n,2) = n+1$), allows us to almost completely determine $f(n,k)$ for $n \le 6$. We were able to fill in the few outstanding values through a computer search (using SageMath~\cite{sage} and Gurobi~\cite{gurobi}).\footnote{Some of these values we first proved by hand, via direct case analysis. However, as we do not see any more broadly applicable generalisation of the arguments therein, we have omitted these proofs.}

\begin{table}[h]
\centering
\begin{tabular}{c|c|c|c|c|c|c|c|c|c|c|c|c|c|c|c}
\backslashbox{n}{k} & 3 & 4  & 5  & 6  & 7  & 8 & 9 & 10 & 11 & 12 & 13 & 14 & 15 & 16 & $\cdots$\\[0.1cm]
\hline
& & & & & & & & & & & & & & & \\[-0.3cm]
3 & \green{6*} & \green{7} & \green{9}  & \green{11} & \green{13} & \green{14} & \green{16} & \green{18} & \green{20} & \green{21} & \green{23} & \green{25} & \green{27} & \green{28} & \green{$\cdots$}\\[0.1cm]
\hline
& & & & & & & & & & & & & & & \\[-0.3cm]
4 & \blue{7*} & \green{8}  & \green{10} & \green{12} & \green{14} & \green{15} & \green{17} & \green{19} & \green{21} & \green{23} & \green{25} & \green{27} & \green{29} & \green{30} & \green{$\cdots$}\\[0.1cm]
\hline
& & & & & & & & & & & & & & & \\[-0.3cm]
5 & \blue{8*} & \red{10*} & \orange{11} & \blue{13} & \blue{15} & \green{16} & \green{18} & \green{20} & \green{22} & \green{24} & \green{26} & \green{28} & \green{30} & \green{31} & \green{$\cdots$} \\[0.1cm] 
\hline
& & & & & & & & & & & & & & & \\[-0.3cm]
6 & \blue{9*}  &  \blue{11*}  & \red{13*} &  \orange{14}  & \blue{16} & \red{18}  & \red{20} & \red{22} & \red{23} & \blue{25} & \blue{27} & \blue{29} & \blue{31} & \green{32} & \green{$\cdots$} \\[0.1cm] 

\end{tabular}
\caption{$f(n,k)$ for $3 \le n \le 6$: values in \green{green} come from Theorem~\ref{thm:main}(a), values in \blue{blue} are a consequence of the recursive bounds, values in \orange{orange} follow from Proposition~\ref{prop:g2diagonal}, and values in \red{red} were obtained by a computer search. An asterisk denotes values equal to the upper bound of Lemma~\ref{lem:genupperbound}; that is, where $n \ge n_0(k)$.
}
\label{tab:uptosix}
\end{table}

\subsection{The extended Golay code}

We see from Table~\ref{tab:uptosix} that $n_0(k) = k+1$ for $k \in \{4,5\}$, leading some credence to the belief that the construction from Proposition~\ref{prop:g2diagonal} is perhaps indeed the last time the upper bound from Lemma~\ref{lem:genupperbound} can be improved. However, we can once again exploit the coding theory connection of Proposition~\ref{prop:strictcodeequiv} to show that this is not always the case.

The extended binary Golay code is a $12$-dimensional code of length $24$ and minimum distance $8$. By Proposition~\ref{prop:strictcodeequiv}, this code is equivalent to an $(8,1;0)$-cover of $\F{2}{12}$ of size $24$, thus implying that $f(12,8)\leq 24$, whereas the upper bound given by Lemma~\ref{lem:genupperbound} is $25$. Furthermore, we see in Table~\ref{tab:uptosix} that $f(6,8) = 18$. By repeated application of Lemma~\ref{lem:recursive}, we must have $f(12,8) \ge f(6,8) + 6$, and thus $f(12,8) = 24$. Moreover, there must be equality in every step of the recursion, and thus $f(n,8) = n+12$ for $6 \le n \le 12$.

This result, coupled with the techniques described previously, allows us to extend Table~\ref{tab:uptosix} to include values for $7 \le n \le 12$ and $3 \le k \le 10$. These new values are depicted in Table~\ref{tab:small} below. We see that the equality $n_0(k) = k+1$ persists for $k = 6, 7$ until the Golay construction comes into existence. In light of Lemma~\ref{lem:krecursion}, this ensures $n_0(k) \ge k+2$ for $8 \le k \le 11$.

\begin{table}[h]
\centering
\begin{tabular}{c|c|c|c|c|c|c|c|c}
\backslashbox{n}{k} & 3 & 4  & 5  & 6  & 7  & 8 & 9 & 10 \\[0.1cm]
\hline
& & & & & & & & \\[-0.3cm]
6 & \blue{9*}  &  \blue{11*}  & \red{13*} &  \orange{14}  & \blue{16} & \red{18}  & \red{20} & \red{22} \\[0.1cm] 
\hline
& & & & & & & &  \\[-0.3cm]
7 & \blue{10*}  & \blue{12*} & \blue{14*} & \red{16*}  &  \orange{17}  & \blue{19} & \blue{21} & \blue{23}  \\[0.1cm]
\hline
& & & & & & & &  \\[-0.3cm]
8 & \blue{11*}  & \blue{13*} & \blue{15*} & \blue{17*}  & \red{19*}   &  \orange{20} & \blue{22} & \blue{24} \\[0.1cm]
\hline
& & & & & & & &  \\[-0.3cm]
9 & \green{12*}  & \blue{14*} & \blue{16*} & \blue{18*}  & \blue{20*} & \cyan{21} & \orange{23} & \blue{25} \\[0.1cm]
\hline
& & & & & & & &  \\[-0.3cm]
10 & \green{13*}  & \blue{15*} & \blue{17*} & \blue{19*}  & \blue{21*}   &  \cyan{22} & \cyan{24} & \orange{26} \\[0.1cm]
\hline
& & & & & & & &  \\[-0.3cm]
11 & \green{14*}  & \blue{16*} & \blue{18*} & \blue{20*}  & \blue{22*}   & \cyan{23}  & \cyan{25} & \cyan{27} \\[0.1cm]
\hline 
& & & & & & & &  \\[-0.3cm]
12 & \green{15*}  & \blue{17*} & \blue{19*} & \blue{21*}  & \blue{23*}   & \cyan{24}  & \cyan{26} & \cyan{28} \\[0.1cm]
\hline 
\vdots & \green{\vdots}  & \blue{\vdots}  & \blue{\vdots}  & \blue{\vdots}  &  \blue{\vdots}  &    &   & \\
\end{tabular}
\caption{More values of $f(n,k)$:
\green{green} represents values coming from Theorem~\ref{thm:main}(a), \red{red} represents values obtained through computer computations, \blue{blue} represents values obtained from other values by the recursive bounds, \orange{orange} represents values obtained by Proposition~\ref{prop:g2diagonal} and recursion, and \cyan{cyan} represents values obtained by the Golay code construction and its recursive consequences. An asterisk denotes values attaining the upper bound of Lemma~\ref{lem:genupperbound}; that is, where $n \ge n_0(k)$.}
\label{tab:small}
\end{table}

This begs the question of what happens for larger values of $k$. Does the gap $n_0(k) - k$ continue to grow? Does the threshold return to $k+1$ at a later point? Unlike the construction in Proposition~\ref{prop:g2diagonal}, the Golay code yields a sporadic construction, which we have not been able to generalise. Furthermore, this is known as a particularly efficient code, and we are not aware of any other code whose parameters lead to an improvement on Proposition~\ref{prop:g2diagonal}. Hence, we are leaning towards the second possibility -- not strongly enough, perhaps, to conjecture it as the truth, but enough to pose it as a question.

\begin{ques} \label{que:k+1threshold}
Do we have $n_0(k) = k+1$ for all $k \ge 12$?
\end{ques}

To answer Question~\ref{que:k+1threshold}, we need to determine the value of $f(k+1,k)$. For an affirmative answer, we need to show $f(k+1,k) = 3k-2$, while a negative answer would follow from a construction showing $f(k+1,k) \le 3k-3$. What could such a construction look like? If we retrace the proof of Theorem~\ref{thm:main}(b), we see that any $k$-cover of $\F{2}{k+1}$ that covers the origin at least $k-2$ times must have size at least $3k-2$. Hence, any construction negating Question~\ref{que:k+1threshold} must cover the origin at most $k-3$ times.

While this seemingly contradicts Corollary~\ref{cor:highcovering}, recall that we needed $n$ to be exponentially large with respect to $k$ to draw that conclusion. Without this condition, the Hamming bound on codes with large distance is not strong enough to provide the requisite lower bound on $f(n,k)$. Indeed, the Gilbert-Varshamov bound, discussed in Remark~\ref{rem:gvbound}, shows that a random collection of $k + O(\log k)$ hyperplanes forms a $3$-cover of $\F{2}{k+1}$ with high probably. Adding $k-3$ arbitrary pairs of parallel planes then gives a $k$-cover of size $3k + O(\log k)$ that only covers the origin $k-3$ times. Thus, we can find numerous $k$-covers that are asymptotically optimal, and we cannot hope for any strong stability when $n$ and $k$ are comparable.


\section{Concluding remarks} \label{sec:conc}

In this paper, we investigated the minimum number of affine subspaces of a fixed codimension needed to cover all nonzero points of $\F{2}{n}$ at least $k$ times, while only covering the origin at most $k-1$ times. We were able to determine the answer precisely when $k$ is large with respect to $n$, or when $n$ is large with respect to $k$, and provided asymptotically sharp bounds for the range in between these extremes. In this final section, we highlight some open problems and avenues for further research.

\paragraph{Bounding the threshold} 

In the previous section, we raised the question of determining the threshold $n_0(k)$ beyond which the result of Theorem~\ref{thm:main}(b) holds. Although our proof requires $n$ to be exponentially large with respect to $k$, our constructions suggest the threshold might, with limited exceptions, be as small as $k+1$.

It is quite possible that solving Question~\ref{que:k+1threshold} will require improving the classic bounds on the length of binary codes of large minimum distance, and will therefore perhaps be quite challenging. However, there is plenty of scope to attack the problem from the other direction, and aim to reduce the exponential upper bound on $n_0(k)$.

Our strategy was to prove the lower bound for $g(n,k,1;k-1)$ and $g(n,k,1;k-2)$, using the recursive bounds. By removing planes covering the origin, we could reduce the remaining cases to $g(n,3,1;0)$, for which, when $n$ is large, the coding theory connection provides a large enough lower bound.

There are two natural ways to improve this argument. The first would be to extend the values $s$ for which we directly prove the lower bound on $g(n,k,1;s)$. For instance, if we could show that $g(n,k,1;s) \ge n + 2k - 3$ for $s \in \{k-3, k-4\}$ as well, then we could reduce the remaining cases to $g(n,5,1;0)$ instead, for which the Hamming bound gives a stronger lower bound. This would still yield an exponential bound on $n_0(k)$, but with a smaller base.

The second approach concerns our reduction to $g(n,3,1;0)$, where we use the fact that removing a hyperplane from a $k$-cover leaves us with a $(k-1)$-cover. However, our constructions contain arbitrary pairs of parallel planes, and thus it is possible to remove from them \emph{two} planes and still be left with a $(k-1)$-cover. If we can show that this is true in general, it could lead to a linear bound on $n_0(k)$.

Finally, while we have focused on the hyperplane case in Question~\ref{que:k+1threshold}, it would also be worth exploring the corresponding threshold $n_0(k,d)$ for $d \ge 2$. It would be very interesting if there were new constructions that appear in this setting where we cover with affine subspaces of codimension $d$.

\paragraph{Larger fields}

In this paper we have worked exclusively over the binary field $\F{2}{}$, but it is also natural to explore these subspace covering problems over larger finite fields, $\mathbb{F}_q$ for $q > 2$. Let us denote the corresponding extremal function by $f_q(n, k, d)$, which is the minimum cardinality of a multiset of $(n - d)$-dimensional affine subspaces that cover all points of $\mathbb{F}_q^n \setminus \{\vec{0}\}$ at least $k$ times, and the origin at most $k - 1$ times. The work of Jamison~\cite{J77} establishes the initial values of this function, showing $f_q(n, 1, d) = (q - 1)(n - d) + q^d - 1$. When it comes to multiplicities $k \ge 2$, some of what we have done here can be transferred to larger fields as well.

To start, we can once again resolve the setting where the multiplicity $k$ is large with respect to the dimension $n$. Indeed, the double-counting lower bound of Lemma~\ref{lem:double-count} generalises immediately to this setting, giving $f_q(n, k, d) \geq q^d k - \floor*{ \frac{k}{q^{n - d}} }$, and one can obtain a matching upper bound by taking multiple copies of every affine subspace.

In the other extreme, where $n$ is large with respect to $k$, the problem remains widely open. We first note that the reduction to hyperplanes from Lemma~\ref{lem:hyperplanerecursion} can be extended, giving $f_q(n, k, d) \leq f_q( n - d + 1, k, 1) + (q^{d - 1} - 1)kq$. Thus, as before, it is best to first focus on the case $d = 1$, and we define $f_q(n, k) \coloneqq f_q(n, k, 1)$. Then Jamison's result gives $f_q(n,1) = (q-1)n$.

For an upper bound, let us start by considering $2$-covers. It is once again true that if one takes the standard $1$-covering by hyperplanes, consisting of all hyperplanes of the form $\{\vec{x}: x_i = c \}$ for some $i \in [n]$ and $c \in \F{q}{} \setminus \{0\}$, the only nonzero vectors that are only covered once are those of Hamming weight $1$. However, since the nonzero coordinate of these vectors can take any of $q-1$ different values, it takes a further $q-1$ hyperplanes to cover these again, and so we have $f(n,2) \le (q-1)(n+1)$. Now, given a $(k-1)$-cover of $\F{q}{n}$, one can obtain a $k$-cover by adding an arbitrary partition of $\F{q}{n}$ into $q$ parallel planes, and this yields $f_q(n,k) \le (q-1)(n+1) + q(k-2)$. This construction is the direct analogue of that from Lemma~\ref{lem:genupperbound}, and so, as in Theorem~\ref{thm:main}(b), we expect it to be tight when $n$ is sufficiently large.

However, the lower bounds are lacking. A simple general lower bound is obtained by noticing that removing $k-1$ hyperplanes from a $k$-cover leaves us with at least a $1$-cover, and so $f_q(n,k) \ge f_q(n,1) + k-1 = (q-1)n + k - 1$. This remains the best lower bound we know --- in particular, even the case of $f_q(n,2)$ is unsolved.

It would of course be very helpful to use some of the machinery we have developed here, and so we briefly explain where the difficulties therein lie. Key to our binary proof was the equivalence with codes of a certain minimum distance, given in Proposition~\ref{prop:strictcodeequiv}. When working over $\F{q}{}$, unfortunately, that equivalence breaks down. For an $n$-dimensional linear code with minimum distance $k$ with generator matrix $A$, we require that, for every nonzero vector $\vec{x} \in \F{q}{n}$, the vector $A \vec{x}$ has at least $k$ nonzero entries. In the binary setting, this was precisely what we wanted, since $\vec{x}$ was covered by the $i$th hyperplane if and only if the $i$th entry of $A \vec{x}$ was nonzero. However, in the $q$-ary setting, for $\vec{x}$ to be covered by the $i$th hyperplane, we need the $i$th entry of $A \vec{x}$ to be equal to a prescribed nonzero value. Hence, while every $k$-covering of $\F{q}{n}$ gives rise to a linear $q$-ary $n$-dimensional code of minimum distance at least $k$, the converse is not true. As a result, the coding theoretic bounds, which are of the form $n + O(k \log n)$, are not strong enough to give us information here.

Another main tool was the recursion over $n$, showing that $f(n,k)$ is strictly increasing in $n$. The same proof goes through here, and we can again show $f_q(n,k) > f_q(n-1,k)$. However, from our bounds, we expect the stronger inequality $f_q(n,k) \ge f_q(n-1,k) + q - 1$ to hold. Intuitively, this is because when we restrict a $k$-cover of $\F{q}{n}$ to $\F{q-1}{n} \subset \F{q}{n}$, there are $q-1$ affine copies of $\F{q-1}{n}$ that are lost. However, this does not (appear to) come out of our probabilistic argument.

It would thus be of great interest to develop new tools to handle the $q$-ary case, as these may also bear fruit when applied to the open problems in the binary setting as well. We believe that new algebraic ideas may be necessary to resolve the following question.

\begin{ques}
For $n \ge n_0(k,q)$, do we have $f_q(n,k) = (q-1)(n+1) + q(k-2)$?
\end{ques}

\paragraph{Polynomials with large multiplicity}

Finally, speaking of algebraic methods, we return to our introductory discussion of the polynomial method. Recall that previous lower bounds in this area have been obtained by considering the more general problem of the minimum degree of a polynomial in $\F{}{}[x_1, x_2, \hdots, x_n]$ that vanishes with multiplicity at least $k$ at all nonzero points in some finite grid, and with lower multiplicity at the origin. Sauermann and Wigderson's recent breakthrough, Theorem~\ref{thm:sauwig}, resolves this polynomial problem for $n \ge 2k-3$ over fields of characteristic 0, while our results here show that, in the binary setting at least, there is separation between the hyperplane covering and the polynomial problems.

Despite this, we wonder whether the answers to the two problems might coincide in the range where the multiplicity $k$ is large with respect to the dimension $n$. That is, can the simple double-counting hyperplane lower bound be strengthened to the polynomial setting? We would therefore like to close by emphasising a question of Sauermann and Wigderson~\cite{SW20}, this time over $\F{2}{}$.

\begin{ques}
Given positive integers $k, n$ with $k \ge 2^{n-2}$, let $P \in \F{2}{}[x_1, x_2, \hdots, x_n]$ be a polynomial that vanishes with multiplicity at least $k$ at every nonzero point, and with multiplicity at most $k-1$ at the origin. Must we then have $\deg( P ) \ge 2k - \floor*{\frac{k}{2^{n-1}}}$?
\end{ques}


\end{document}